\pdfoutput=1
\documentclass[11pt]{amsart}
\usepackage{amsmath}
\usepackage{amssymb}
\usepackage{mathtools}

\usepackage{enumerate}
\usepackage{slashed}
\usepackage{caption}
\usepackage{subcaption}
\usepackage{newlfont}
\usepackage{amsrefs}
\usepackage{comment}
\usepackage[normalem]{ulem}
\usepackage{accents}
\usepackage{leftidx}

\usepackage{harpoon}
\usepackage[pdftex]{graphicx}
\usepackage{esint}
\usepackage{relsize}
\usepackage[latin1]{inputenc}
\usepackage{xcolor}

\definecolor{brass}{rgb}{0.71, 0.65, 0.36}

\usepackage{tikz-cd}
\theoremstyle{plain}
\newtheorem{theorem}{Theorem}[section]
\newtheorem{lemma}[theorem]{Lemma}

\newtheorem{proposition}[theorem]{Proposition}

\theoremstyle{definition}
\newtheorem{remark}[theorem]{Remark}

\numberwithin{equation}{section}

\def\inf{\operatorname{inf}}

\theoremstyle{plain}

\numberwithin{equation}{section}

\allowdisplaybreaks

\begin{document}

\title[A Quasi-Local Mass]{A Quasi-Local Mass}

\author[Alaee]{Aghil Alaee}
\address{
Department of Mathematics, Clark University, Worcester, MA 01610, USA;
Center of Mathematical Sciences and Applications, Harvard University, Cambridge, MA. 02138, USA}
\email{aalaeekhangha@clarku.edu, aghil.alaee@cmsa.fas.harvard.edu}

\author[Khuri]{Marcus Khuri}
\address{Department of Mathematics\\
Stony Brook University\\
Stony Brook, NY 11794, USA}
\email{khuri@math.sunysb.edu}

\author[Yau]{Shing-Tung Yau}
\address{Yau Mathematical Sciences Center, Tsinghua University, Beijing, 100084, China;
Yanqi Lake Beijing Institute of Mathematical Sciences and Applications, Beijing, 101408, China}
\email{styau@tsinghua.edu.cn, yau@math.harvard.edu}


\thanks{A. Alaee acknowledges the support of NSF Grant DMS-2316965, and an AMS-Simons travel grant. M. Khuri acknowledges the support of NSF Grant DMS-2104229.}

\begin{abstract}
We define a new gauge independent quasi-local mass and energy, and show its relation to the Brown-York Hamilton-Jacobi analysis. A quasi-local proof of the positivity, based on spacetime harmonic functions, is given for admissible closed spacelike 2-surfaces which enclose an initial data set satisfying the dominant energy condition. Like the Wang-Yau mass, the new definition relies on isometric embeddings into Minkowski space, although our notion of admissibility is different from that of Wang-Yau. Rigidity is also established, in that vanishing energy implies that the 2-surface arises from an embedding into Minkowski space, and conversely the mass vanishes for any such surface. Furthermore, we show convergence to the ADM mass at spatial infinity, 
and provide the equation associated with optimal isometric embedding.
\end{abstract}

\maketitle

\section{Introduction}
\label{sec1} \setcounter{equation}{0}
\setcounter{section}{1}

In 1982, Penrose \cite{PenroseR} listed a set of major open problems in which one was to \emph{find a suitable 
quasi-local definition of energy-momentum in general relativity}. As is well-known, the fundamental difficulty 
is that there is no natural notion of energy density for the gravitational field, due to Einstein's principle of equivalence. This has led to a plethora of different mathematical formulations including: the localization of ADM mass by Bartnik \cite{Bartnik}, twistor and spinor approaches of Dougan-Mason \cite{DM}, Ludvigsen-Vickers \cite{LV}, Penrose \cite{Penrosem}, and Zhang \cite{xzhang}, the Hawking mass \cite{Hawking}, as well as Hamilton-Jacobi methods employed by Booth-Mann \cite{BM}, Brown-York \cite{BrownYork}, Epp \cite{Epp}, Hawking-Horowitz \cite{HH}, Kijowski \cite{Kijowski}, Liu-Yau \cite{LiuYau}, and Wang-Yau \cite{WangYau}. For a detailed account of the various quasi-local masses, see the survey paper of Szabados \cite{Szabados}.

There are several desirable properties that a quasi-local mass should have. For instance, like the positive mass theorem it should be nonnegative for a large class of surfaces in the presence of the dominant energy condition, and exhibit rigidity in the sense that it vanishes if and only if the surface arises from Minkowski space. Appropriate asymptotics are also important, in that the ADM and Bondi masses should be recovered in the large sphere limits at spatial and null infinities. Moreover, gauge independence and relation with a Hamilton-Jacobi analysis are preferable, in order to aid with physical relevance and interpretations. While many other properties of the quasi-local mass itself may be added to the list, a potentially advantageous characteristic of a proof of nonnegativity is that it be quasi-local. This refers to a proof strategy that appeals solely to a compact initial data set enclosed by the spacelike 2-surface in question, as opposed to the use of asymptotically flat (or other) extensions and the positive mass theorem. The latter strategy of extensions is essential in the work of Shi-Tam \cite{ShiTam} on the Brown-York mass, as well as for the Liu-Yau \cite{LiuYau} and Wang-Yau \cite{WangYau0,WangYau} masses, and is even contained within the definition of the Bartnik mass \cite{Bartnik}. In fact, Schoen asked in \cite{Schoen} whether it is possible to \textit{find a quasi-local proof of nonnegativity for the Brown-York and related quasi-local masses}. An affirmative answer to this question for the Brown-York definition has recently been put forward by Montiel \cite{Montiel}, utilizing a spinorial approach. 

In the current paper, we introduce a new gauge independent expression for quasi-local mass with a quasi-local proof of nonnegativity. This mass also comes with a Hamilton-Jacobi interpretation, and although it shares some similarities with the Wang-Yau definition, the new notion has a completely separate derivation and involves a different range of applicability. The derivation and basic motivation comes from certain integral expressions associated with
spacetime harmonic functions \cite{HKK}. In \cite{Stern}, level set techniques for harmonic maps to $S^1$ were introduced into the study of scalar curvature on closed 3-dimensional Riemannian manifolds. Inspired by this, a proof of the Riemannian positive mass theorem \cite{BKKS} was given based on level sets of asymptotically linear harmonic functions. Spacetime harmonic functions were then introduced \cite{HKK} to treat the asymptotically flat spacetime version of the positive mass theorem, and have subsequently been used to prove the corresponding theorem in asymptotically (locally) hyperbolic settings \cite{AlaeeKhuriHung,BHKKZ1}, as well as for comparison theorems in Riemannian geometry \cite{HKKZ,HKKZ1}; surveys of some recent advancements may be found in \cite{BHKKZ1,Stern1}. Throughout this work, unless specified otherwise, all manifolds will be assumed to be connected, oriented, and smooth.

Let $\Sigma$ be a closed (compact without boundary) spacelike 2-surface having induced metric $\sigma$ in spacetime $N^{3,1}$ with metric $\langle\cdot,\cdot\rangle$ of signature $(-+++)$, and let $\{e_3,e_4\}$ be an orthonormal frame for the normal bundle consisting of a spacelike and future directed timelike vector respectively. In this gauge, the normal bundle connection 1-form and mean curvature vector are given by
\begin{equation}
\alpha_{e_3}(\cdot)=\langle\nabla^N_{(\cdot)}e_3,e_4\rangle,\quad\quad\quad \vec{H}=(\mathrm{div}_{\sigma}e_3)e_3 -(\mathrm{div}_{\sigma}e_4)e_4.
\end{equation}
Consider an isometric embedding $\iota:\Sigma\hookrightarrow \mathbb{R}^{3,1}$, and choose Cartesian coordinates $(\mathbf{t},\mathbf{x}^i)$, $i=1,2,3$ for the target Minkowski space. One may then obtain a function $u_{a}=\iota^{*}\left(-\mathbf{t}+a_i\mathbf{x}^i\right)$ on $\Sigma$ for any set of constants $\mathbf{a}=(a_1,a_2,a_3)$, which will be restricted to satisfy $|\mathbf{a}|^2=\sum_i a_i^2=1$. Recall that in \cite{WangYau0,WangYau}, Wang-Yau parameterize a class of isometric embeddings into Minkowski space with a time function $\tau$ on $\Sigma$ satisfying the convexity condition
\begin{equation}\label{convexity}
\left(1+|\nabla_\partial\tau|^2\right)K_{\tilde{\sigma}}=K_{\sigma}
+\frac{\text{det}(\nabla_\partial^2\tau)}{(\det\sigma)\left(1+|\nabla_\partial\tau|^2\right)}>0,
\end{equation}
where $K_{\sigma}$ and $K_{\tilde{\sigma}}$ denote the Gaussian curvatures of $\sigma$ and $\tilde{\sigma}=\sigma+d\tau^2$, and $\nabla_\partial$ is the connection with respect to $\sigma$. By the classical theorem of Nirenberg \cite{Nirenberg} and Pogorelov \cite{Pogorelov} there exists a unique isometric embedding up to rigid motion into $\mathbb{R}^3$, and from this one obtains an isometric embedding into $\mathbb{R}^{3,1}$.  An alternative method to produce isometric embeddings, which does not rely on \eqref{convexity}, is to embed $\Sigma$ into a (hyperboloid) hyperbolic space $\mathbb{H}^3_{-\kappa}\subset\mathbb{R}^{3,1}$ with large $\kappa>0$ to aid with ellipticity. For instance, every metric on the 2-sphere admits such an isometric embedding \cite{Pogorelov1,Pogorelov2}, and a related result of Gromov \cite[Section 3.2.4]{Gromov} shows that any closed 2-surface admits an isometric embedding into some complete 3-dimensional Riemannian target space of constant negative sectional curvature.

The quintuple $(\sigma,\vec{H},\alpha,\iota,u_a)$ will be used to build the quasi-local energy, and may be referred to as a \textit{quasi-local data set} for $\Sigma$. If the mean curvature vector is spacelike, then for any constant $\varepsilon>0$, there exists a unique orthonormal frame $\{\bar{e}_3,\bar{e}_4\}$ for the normal bundle of $\Sigma$ such that
\begin{equation}\label{canonicalframe}
\langle\vec{H},\bar{e}_3\rangle>0,\qquad \quad \langle\vec{H},\bar{e}_4\rangle=\frac{-\Delta_\partial u_a}
{\sqrt{|\nabla_\partial u_a|^2+\varepsilon^2}}.
\end{equation}
We define the quasi-local energy with respect to the observer determined by the pair $(\iota,u_a)$ to be
\begin{equation}\label{qlenergy}
E(\Sigma,\iota,u_a)
=\lim_{\varepsilon\to 0}\frac{1}{8\pi}\int_{\Sigma}\left(\mathcal{H}_0(\hat{e}_3,u_a)
-\mathcal{H}(\bar{e}_3,u_a)\right)dA,
\end{equation}
where the \textit{quasi-local Hamiltonian density} is
\begin{equation}
\mathcal{H}(\bar{e}_3,u_a)
=\sqrt{|\nabla_\partial u_a|^2+\varepsilon^2}\,\langle\vec{H},\bar{e}_3\rangle
+\alpha_{\bar{e}_3}(\nabla_\partial u_a)
\end{equation}
and $\mathcal{H}_0$ represents the same quantity for $\iota(\Sigma)\subset\mathbb{R}^{3,1}$ with $\{\hat{e}_3,\hat{e}_4\}$ denoting the frame given by equations \eqref{canonicalframe} in the Minkowski setting. In analogy with special relativity, the quasi-local mass is then set to be the infimum of energy over all admissible observers
\begin{equation}
\text{mass}(\Sigma)=\inf_{(\iota,u_a)}E(\Sigma,\iota,u_a).
\end{equation}
Due to the null trajectory of the observers, this mass may be interpreted as the difference of energy and norm of linear momentum from the 4-momentum vector as opposed to its Lorentz length; see the remark in \cite[pg. 1496]{Chenetal} for a related discussion of this discrepancy. 

A pair $(\iota,u_a)$ will be called \textit{admissible} for $\Sigma\subset N^{3,1}$ if the mean curvature vector of $\iota(\Sigma)$ is spacelike, and the regular level sets of $u_a$ are maximal in the following sense. Namely, there exists a compact spacelike hypersurface $\hat{\Omega}\subset\mathbb{R}^{3,1}$ with $\partial\hat{\Omega}=\iota(\Sigma)$, such that if any regular $s$-level set of $u_a$ has $n$ components then the Euler characteristic of the $s$-level fill-in with respect to $\hat{\Omega}$ satisfies $\chi(\hat{\Sigma}_s)=n$. Here, $\hat{\Sigma}_s$ denotes the set of points within $\hat{\Omega}$ satisfying the level set equation $s=-\mathbf{t}+a_i\mathbf{x}^i$. Furthermore, we will say that the \textit{dominant energy condition} holds for a spacetime if the Einstein equations are enforced with stress-energy tensor satisfying $T(\mathbf{v},\mathbf{w})\geq 0$ for all future-pointing causal vectors $\mathbf{v}$ and $\mathbf{w}$.

\begin{theorem}\label{thm1} 
Let $\Sigma$ be a closed spacelike 2-surface in a 4-dimensional spacetime satisfying the dominant energy condition, and assume that $\Sigma$ has an outward pointing spacelike mean curvature vector and bounds a compact spacelike hypersurface $\Omega$ with trivial second homology $H_2(\Omega;\mathbb{Z})=0$. The following statements hold with $\mathbf{a}\in\mathbb{R}^3$ satisfying $|\mathbf{a}|=1$.
\begin{enumerate}
\item 
If an isometric embedding $\iota:\Sigma\hookrightarrow\mathbb{R}^{3,1}$ together with function $u_a$ is admissible for $\Sigma$, then the quasi-local energy limit exists, is finite, and is nonnegative: $E(\Sigma,\iota,u_a)\geq 0$. 

\item
Under the same hypotheses, if $E(\Sigma,\iota,u_a)=0$ for all $\mathbf{a}$ then $(\Sigma,\sigma,\vec{H},\alpha)$ arises from Minkowski space. 

\item
If $\Sigma\subset\mathbb{R}^{3,1}$ is such that the inclusion map is admissible with some $u_a$, then the quasi-local mass vanishes: $\text{mass}(\Sigma)=0$.
\end{enumerate}
\end{theorem}

\begin{remark}\label{remark1}
Nonnegativity of the energy holds in the more general circumstance of a disconnected $\Sigma$ and without the homology assumption on $\Omega$, when a suitable modification of the admissibility condition is enforced. This is discussed in more detail below at the end of Section \ref{sec4}.
\end{remark}

A model situation for an admissible pair $(\iota,u_a)$ occurs when $\Sigma$ is a topological 2-sphere, and its isometric image in Minkowski space bounds a 3-ball whose generic intersection with the null hyperplanes associated to $u_a$ is in the form of a disc. Therefore, this theorem yields positivity under the dominant energy condition for a large class of surfaces.
Moreover, because this admissibility condition is not based on a convexity property such as \eqref{convexity} which appears in \cite[Definition 5.1]{WangYau}, it is applicable beyond spheres to surfaces of positive genus. Beyond this difference with the Wang-Yau mass in terms of the admissibility conditions and range of applicability, there are two other immediate distinguishing characteristics between these two quasi-local masses. First, as explained in Section \ref{sec3}, a Hamilton-Jacobi analysis shows that energy \eqref{qlenergy} is measured by a null observer while the Wang-Yau energy is measured by a timelike observer. Secondly, while the proof of nonnegativity for the Wang-Yau mass relies on asymptotically flat extensions and ultimately the positive mass theorem, the proof here requires only quasi-local information.


A consequence of the proof of Theorem \ref{thm1} is that the limit may be evaluated explicitly in definition \eqref{qlenergy}. In particular, if $\tilde{\Sigma}$ denotes the open subset of $\Sigma$ on which $|\nabla_{\partial} u_a|\neq 0$ then we have
\begin{align}
\begin{split}
E(\Sigma,\iota,u_a)=&\frac{1}{8\pi}\int_{\tilde{\Sigma}}\left(\sqrt{|\vec{H}_0|^2|\nabla_\partial u_a|^2+(\Delta_\partial u_a)^2}+\nabla_{\partial}f_0\cdot\nabla_{\partial}u_a+\alpha_{\frac{\vec{H}_0}{|\vec{H}_0|}}(\nabla_{\partial}u_a)\right)dA\\
&-\frac{1}{8\pi}\int_{\tilde{\Sigma}}\left(\sqrt{|\vec{H}|^2|\nabla_\partial u_a|^2+(\Delta_\partial u_a)^2}+\nabla_{\partial}f\cdot\nabla_{\partial}u_a+\alpha_{\frac{\vec{H}}{|\vec{H}|}}(\nabla_{\partial}u_a)\right) dA,
\end{split}
\end{align}
where $\vec{H}_0$ is the mean curvature vector of the isometric embedding $\iota(\Sigma)$ and
\begin{equation}
f_{0}=\sinh^{-1}\left(\frac{\Delta_{\partial}u_a}{|\vec{H}_0||\nabla_{\partial}u_a|}\right),\qquad f=\sinh^{-1}\left(\frac{\Delta_{\partial}u_a}{|\vec{H}||\nabla_{\partial}u_a|}\right).    
\end{equation}

In addition to the attributes of positivity and rigidity, the new definition also behaves well with respect to asymptotic limits at spatial infinity. Let $(M,g,k)$ be a 3-dimensional initial data set for the Einstein equations, with $g$ denoting a Riemannian metric and $k$ a symmetric 2-tensor representing the extrinsic curvature in spacetime.
These objects satisfy the constraint equations
\begin{equation}
\mu=\frac{1}{2}\left(R_g +(\mathrm{Tr}_g k)^2 -|k|^2\right),\quad\quad
J=\mathrm{div}_g\left(k-(\mathrm{Tr}_g k)g\right),
\end{equation}
where $R_g$ is the scalar curvature and $\mu$ and $J$ represent the energy and momentum density of matter fields multiplied by $8\pi$. The data are asymptotically flat (with one end) if outside a compact set, $M$ is diffeomorphic to the compliment of a ball in Euclidean space $\mathbb{R}^3 \setminus B_1$, and in the coordinates provided by this diffeomorphism
\begin{equation}\label{asymflat}
|\partial^l (g_{ij}-\delta_{ij})(x)|=O(|x|^{-\tau-l})\quad\text{ } l=0,1,2,3,\quad\quad
|\partial^l k_{ij}(x)|=O(|x|^{-\tau-1-l})\quad\text{ } l=0,1,
\end{equation}
for some $\tau>\tfrac{1}{2}$. An extra third derivative of $g$ is included for control of isometric embeddings in the next result. The energy and momentum densities will be taken to be integrable $\mu, J \in L^1(M)$ so that the ADM energy and linear momentum are well-defined 
and given by
\begin{equation}
\mathcal{E}=\lim_{r\rightarrow\infty}\frac{1}{16\pi}\int_{S_{r}}\sum_i \left(g_{ij,i}-g_{ii,j}\right)\upsilon^j dA,\quad\quad
\mathcal{P}_i=\lim_{r\rightarrow\infty}\frac{1}{8\pi}\int_{S_{r}} \left(k_{ij}-(\mathrm{Tr}_g k)g_{ij}\right)\upsilon^j dA,
\end{equation}
where $\upsilon$ is the unit outer normal to the coordinate sphere $S_r$ of radius $r=|x|$ and $dA$ denotes its area element. The ADM mass is the Lorentz length of the ADM energy-momentum vector $(\mathcal{E},\mathcal{P})$. If the dominant energy condition is satisfied which implies that $\mu\geq |J|$, then the spacetime positive mass theorem \cite{EHLS,HKK,SchoenYau,Witten} asserts that the ADM energy-momentum is nonspacelike, and characterizes Minkowski space as the unique spacetime having asymptotically flat initial data with vanishing mass; see \cite{Lee} for a detailed account.

\begin{theorem}\label{thm2} 
Let $(M,g,k)$ be an asymptotically flat initial data set for the Einstein equations, and let $\iota_r :S_r \hookrightarrow\mathbb{R}^{3,1}$ denote the (unique up to Euclidean motion) isometric embedding of the $r$-coordinate sphere into 
a constant time slice $\mathbb{R}^3\subset\mathbb{R}^{3,1}$. Then for any $\mathbf{a}$, the asymptotic limit of quasi-local energies is given in terms of the ADM energy and linear momentum by
\begin{equation}
\lim_{r\to\infty}E(S_r,\iota_r,u_a)=\mathcal{E}-\langle \mathbf{a},\mathcal{P}\rangle.
\end{equation}
\end{theorem}

This paper is organized as follows. In the next section the main ideas behind the derivation of the energy will be explained, while in Section \ref{sec3} a physical interpretation will be given in terms of a Hamilton-Jacobi analysis.
Section \ref{sec4} is dedicated to the proof of nonnegativity, and Section \ref{sec5} deals with the rigidity statement of Theorem \ref{thm1}. Moreover, the asymptotic behavior of Theorem \ref{thm2} will be established in Section \ref{sec6}, 
while the equation associated with optimal isometric embedding is obtained in Section \ref{sec7}.

\smallskip

\noindent \textbf{Acknowledgements.} The authors would like to thank Mu-Tao Wang for helpful comments.

\section{Derivation of the Energy}
\label{sec2} \setcounter{equation}{0}
\setcounter{section}{2}

A motivation for the quasi-local energy comes from the level set technique and certain integral formulae involving \textit{spacetime harmonic functions} \cite{HKK}. Consider a compact initial data set $(\Omega,g,k)$. Recall that a function $u\in C^{2}(\Omega)$ is spacetime harmonic if it satisfies the equation
\begin{equation}
\Delta_g u+(\mathrm{Tr}_g k)|\nabla u|=0,
\end{equation}
and note that this arises as the trace of the \textit{spacetime Hessian}
\begin{equation}
\overline{\nabla}_{ij} u:=\nabla_{ij}u+k_{ij}|\nabla u|.
\end{equation}
The following inequality was established in \cite[Proposition 3.1]{AlaeeKhuriHung}.

\begin{proposition}\label{integralidentity}
Let $(\Omega,g,k)$ be a 3-dimensional compact initial data set with smooth boundary $\partial\Omega$, having outward unit normal $\nu$. Let $u:\Omega\to\mathbb{R}$ be a spacetime harmonic function which lies in $C^{2,\varsigma}(\Omega)$, $0<\varsigma<1$, and consider the open subset $\bar{\partial}\Omega$ of the boundary on which $|\nabla_{\partial} u|\neq 0$, where $\nabla_{\partial}u$ is the projection of the full gradient onto the boundary tangent space. 
If $\overline{u}$ and $\underline{u}$ are the maximum and minimum values of $u$ and $\Sigma_s =u^{-1}(s)$, then
\begin{align}\label{ident}
\begin{split}
&\int_{\partial\Omega}\left(k(\nabla_{\partial}u,\nu)-|\nabla u|H-\nu(u)\mathrm{Tr}_{\partial\Omega}k\right)dA\\
&+\int_{\bar{\partial} \Omega}\frac{|\nabla_{\partial}u|}{|\nabla u|}\nabla_{\partial}u\left(\frac{\nu(u)}{|\nabla_{\partial}u|}\right) dA+2\pi\int_{\underline{u}}^{\bar{u}}\chi(\Sigma_s)ds \\
&\geq\int_{\Omega}\left(\frac{1}{2}\frac{|\overline{\nabla}^2 u|^2}{|\nabla u|}+\mu|\nabla u|+J(\nabla u)\right)dV,
\end{split}
\end{align}
where $\chi(\Sigma_s)$ denotes the Euler characteristic, and $H$ is the mean curvature of the boundary with respect to $\nu$.
\end{proposition}

The right-hand side of \eqref{ident} is nonnegative if the dominant energy condition holds, and this suggests investigating the boundary terms in relation to a quasi-local energy. In order to better interpret the boundary expression with regards to spacetime geometry, assume that the data arise from a spacetime $(\Omega,g,k)\hookrightarrow N^{3,1}$ and let $\mathbf{n}$ be the associated unit timelike future directed normal vector field. Consider the frame $\{\nu,\mathbf{n}\}$ for the $SO(1,1)$ normal bundle over $\partial\Omega$, and observe that the mean curvature vector and an auxiliary vector field are given by
\begin{equation}
\vec{H}=H\nu -(\mathrm{Tr}_{\partial\Omega}k)\mathbf{n},\quad\quad\quad \vec{w}=|\nabla u|\nu +\nu(u)\mathbf{n},
\end{equation}
with $|\vec{w}|^2 =|\nabla_{\partial}u|^2$. Therefore, two of the boundary terms combine to form
\begin{equation}\label{0}
\langle\vec{H},\vec{w}\rangle= |\nabla u|H+\nu(u)\mathrm{Tr}_{\partial\Omega}k.  
\end{equation}

The remaining two boundary terms may be interpreted as follows. In the gauge determined by the current frame, the connection 1-form for the normal bundle is
\begin{equation}\label{1}
\alpha_{\nu}(X)=-k(X,\nu)=\langle \nabla_{X}^{N} \nu,\mathbf{n}\rangle,
\end{equation}
where $X$ is any tangent vector field to the boundary surface. Notice that with a change of gauge to the frame
determined by $e_3=a\nu + b\mathbf{n}$, $e_4 =b\nu +a\mathbf{n}$ with $a=\cosh f$, $b=\sinh f$ yields
\begin{equation}\label{2}
\alpha_{e_3}(X)=\langle\nabla_{X}^{N}e_3,e_4\rangle=a^2 \langle\nabla_{X}^{N}\nu,\mathbf{n}\rangle+b^2 \langle\nabla_{X}^{N}\mathbf{n},\nu\rangle 
+X(a)b-X(b)a=\alpha_{\nu}(X)-X(f),
\end{equation}
for any function $f\in C^1(\partial\Omega)$. Below, for simplicity of the discussion, we will \textit{assume that all calculations occur away from critical points} of $u$ restricted to the boundary. Then choosing $X=\nabla_{\partial}u$ and
$f=\sinh^{-1}\left(\nu(u)/|\nabla_{\partial}u|\right)$ produces
\begin{equation}\label{3}
X(f)=\nabla_{\partial}u\left(\sinh^{-1}\left(\frac{\nu(u)}{|\nabla_{\partial}u|}\right)\right)
=\frac{|\nabla_{\partial}u|}{|\nabla u|}\nabla_{\partial}u\left(\frac{\nu(u)}{|\nabla_{\partial}u|}\right).
\end{equation}
Hence, combining \eqref{1}, \eqref{2}, and \eqref{3} shows that
\begin{equation}\label{4}
\alpha_{e_3}(\nabla_{\partial}u)=\alpha_{\nu}(\nabla_{\partial}u)-\nabla_{\partial}u\left(\sinh^{-1}\left(\frac{\nu(u)}{|\nabla_{\partial}u|}\right)\right)
=-k(\nabla_{\partial}u,\nu)-\frac{|\nabla_{\partial}u|}{|\nabla u|}\nabla_{\partial}u\left(\frac{\nu(u)}{|\nabla_{\partial}u|}\right).
\end{equation}

Observe that the vector $\vec{w}$ is proportional to a special case (when $\varepsilon=0$) of the first member of the \textit{level set frame}
\begin{equation}
e'_3 =\frac{\sqrt{|\nabla u|^2+\varepsilon^2}}{\sqrt{|\nabla_{\partial}u|^2+\varepsilon^2}}\,\nu + \frac{\nu(u)}{\sqrt{|\nabla_{\partial}u|^2+\varepsilon^2}}\,\mathbf{n},\qquad e'_4= \frac{\nu(u)}{\sqrt{|\nabla_{\partial}u|^2+\varepsilon^2}}\,\nu +\frac{\sqrt{|\nabla {u}|^2+\varepsilon^2}}{\sqrt{|\nabla_{\partial}u|^2+\varepsilon^2}}\,\mathbf{n}.
\end{equation}
Here, the parameter $\varepsilon>0$ is included to avoid the technical issue of critical points for the restriction of $u$ to $\partial\Omega$. The computations \eqref{0} and \eqref{4} then motivate us to define the quasi-local Hamiltonian density as
\begin{equation}
\mathcal{H}(e'_3,u)
=\sqrt{|\nabla_\partial u|^2+\varepsilon^2}\,\langle\vec{H},e'_3\rangle
+\alpha_{e'_3}(\nabla_\partial u),
\end{equation}
with respect to the level set frame. The expression for quasi-local energy $E(\Sigma,\iota,u_a)$ in \eqref{qlenergy} follows by choosing an `optimal' frame, and comparing to an appropriate Hamiltonian density in the ground state.

\section{A Hamilton-Jacobi Interpretation of the Energy}
\label{sec3} \setcounter{equation}{0}
\setcounter{section}{3}

In this section we will describe the relationship between the quasi-local energy defined in the introduction, and a Hamilton-Jacobi analysis. Recall that Brown-York \cite{BrownYork} and Hawking-Horowitz \cite{HH} derived a Hamiltonian for closed spacelike 2-surfaces $\Sigma\hookrightarrow N^{3,1}$ which enclose a compact initial data set $(\Omega,g,k)$. As before let $\mathbf{n}$ be the unit timelike future directed normal to the data, and let $\nu$ be the unit outer normal to $\Sigma$ with respect to $\Omega$. If $T=\varphi \mathbf{n} +Y$ is a timelike vector field along $\Sigma$, representing an observer with lapse $\varphi$ and shift $Y$, then the surface Hamiltonian takes the form
\begin{equation}
\mathbf{H}(\Sigma,T,\mathbf{n})=-\frac{1}{8\pi}\int_{\Sigma}\left(\varphi H-k(\nu,Y)+(\text{Tr}_{g}k)g(\nu,Y)\right) dA,
\end{equation}
where $dA$ is the area element on $\Sigma$. As pointed out by Wang-Yau \cite{WangYau0}, this Hamiltonian may be reexpressed with the aid of the vector field
\begin{equation}
P=H\mathbf{n}+k(\nu)-(\text{Tr}_{g}k)\nu,
\end{equation} 
so that 
\begin{equation}
\mathbf{H}(\Sigma,T,\mathbf{n})=\frac{1}{8\pi}\int_{\Sigma}\langle P,T\rangle dA.
\end{equation}
Note that $P$ is perpendicular to the the mean curvature vector $\vec{H}=H\nu-(\text{Tr}_\Sigma k) \mathbf{n}$. 
The associated energy is then defined by choosing a reference Hamiltonian, determined by an isometric embedding
$\iota:\Sigma\hookrightarrow\mathbb{R}^{3,1}$ and corresponding vector fields $T_0$ and $\mathbf{n}_0$ along the image in Minkowski space, namely
\begin{equation}
\mathbf{E}(\Sigma)=\mathbf{H}(\Sigma,T,\mathbf{n})-\mathbf{H}_0(\iota(\Sigma),T_0,\mathbf{n}_0).
\end{equation} 

This notion of energy depends on the choices of $T$, $\mathbf{n}$ and $T_0$, $\mathbf{n}_0$, as well as on the isometric embedding. Observe that if the isometric embedding lands in a time slice $\mathbb{R}^3\subset\mathbb{R}^{3,1}$ with normal $\mathbf{n}_0$, and the choices $T=\mathbf{n}$ and $T_0=\mathbf{n}_0$ are made, then the typical expression for the 
the Brown-York mass is recovered; it depends on the initial data $\Omega$ and hence on $\mathbf{n}$. The Liu-Yau mass is obtained with the same prescription, except that $\mathbf{n}$ is taken to satisfy $\langle\vec{H},\mathbf{n}\rangle =0$. Furthermore, given an admissible time function $\tau$ on $\Sigma$ associated with an isometric embedding, the Wang-Yau energy is produced by setting
\begin{equation}
T_0 = \sqrt{1+|\nabla_{\partial}\tau|^2}\mathbf{n}_0 -\nabla_{\partial}\tau,
\end{equation}
and choosing $\mathbf{n}_0$ so that $\{\nu_0,\mathbf{n}_0\}$ is the unique frame for the normal bundle of $\iota(\Sigma)$ 
satisfying
\begin{equation}\label{canonicalframe0}
\langle\vec{H}_0,\nu_0\rangle>0,\qquad \quad \langle\vec{H}_0,\mathbf{n}_0\rangle=\frac{-\Delta_\partial \tau}
{\sqrt{1+|\nabla_\partial \tau|^2}},
\end{equation}    
while $T$ is set to have the same lapse and shift as $T_0$ and $\mathbf{n}$ is chosen to satisfy conditions analogous to \eqref{canonicalframe0} in $N^{3,1}$.

In order to apply these considerations to the quasi-local energy introduced in Section \ref{sec1}, let
$(\mathbf{t},\mathbf{x}^i)$, $i=1,2,3$ be coordinates for the reference Minkowski space, and pull back a linear null function
$u_a=\iota^*(-\mathbf{t}+a_i\mathbf{x}^i)$ to $\Sigma$. For each $\varepsilon>0$ we then set
\begin{equation}
T_0 = \sqrt{|\nabla_{\partial}u_a|^2 +\varepsilon^2}\mathbf{n}_0 +\nabla_{\partial}u_a,
\end{equation}
and choose $\mathbf{n}_0$ so that $\{\nu_0,\mathbf{n}_0\}$ is the unique frame for the normal bundle of $\iota(\Sigma)$ 
satisfying
\begin{equation}\label{canonicalframe2}
\langle\vec{H}_0,\nu_0\rangle>0,\qquad \quad \langle\vec{H}_0,\mathbf{n}_0\rangle=\frac{-\Delta_\partial u_a}
{\sqrt{|\nabla_{\partial}u_a|^2 +\varepsilon^2}}.
\end{equation}    
Moreover, $T$ is set to have the same lapse and shift as $T_0$, and $\mathbf{n}$ is chosen to satisfy conditions analogous to \eqref{canonicalframe2} in $N^{3,1}$. Note that $T$ and $T_0$ are timelike
with $|T|^2=|T_0|^2=-\varepsilon^2$, and are approaching null vectors as $\varepsilon\rightarrow 0$. Observe that writing
$\langle\vec{H},\nu\rangle=H$ and $\langle\vec{H},\mathbf{n}\rangle=\text{Tr}_\Sigma k$ gives rise to
\begin{align}
\begin{split}
\mathbf{H}^{\varepsilon}(\Sigma,T,\mathbf{n}):=&\frac{1}{8\pi}\int_{\Sigma}\langle P,T\rangle dA\\
=&\frac{1}{8\pi}\int_{\Sigma}\left(-H\sqrt{|\nabla_\partial u_a|^2+\varepsilon^2}+k(\nu,\nabla_\partial u_a)\right) dA\\
=&\frac{1}{8\pi}\int_{\Sigma}\left(-\sqrt{|\nabla_\partial u_a|^2+\varepsilon^2}\langle\vec{H},\nu\rangle-\alpha_{\nu}(\nabla_\partial u_a)\right) dA\\
=&-\frac{1}{8\pi}\int_{\Sigma}\mathcal{H}(\nu, u_a)  dA,
\end{split}
\end{align}
and similarly for the reference Hamiltonian. Therefore, the new quasi-local energy arises from Hamiltonians by taking a limit as the observer approaches a null direction
\begin{equation}
E(\Sigma,\iota,u_a)=\lim_{\varepsilon\rightarrow 0}\left(\mathbf{H}^{\varepsilon}(\Sigma,T,\mathbf{n})
-\mathbf{H}^{\varepsilon}_0(\iota(\Sigma),T_0,\mathbf{n}_0)\right),
\end{equation}
where in terms of the notation of the introduction we have $\{\bar{e}_3,\bar{e}_4\}=\{\nu,\mathbf{n}\}$ and
$\{\hat{e}_3,\hat{e}_4\}=\{\nu_0,\mathbf{n}_0\}$.

\section{Proof of Nonnegativity}
\label{sec4} \setcounter{equation}{0}
\setcounter{section}{4}

The purpose of the current section is to establish the inequality portion of Theorem \ref{thm1}. As before, given a linear null function $u_a$ on a spacelike 2-surface $\Sigma$ in spacetime $N^{3,1}$, consider the normal bundle frame 
$\{\bar{e}_3,\bar{e}_4\}$ defined by
\begin{equation}\label{canonicalframe1}
\langle\vec{H},\bar{e}_3\rangle>0,\qquad \quad \langle\vec{H},\bar{e}_4\rangle=\frac{-\Delta_\partial u_a}
{\sqrt{|\nabla_\partial u_a|^2+\varepsilon^2}},
\end{equation}
for $\varepsilon>0$. We begin with a preliminary result, similar to \cite[Proposition 2.1]{WangYau} for the Wang-Yau quasi-local energy, which demonstrates how this may be interpreted as an optimal frame.

\begin{lemma}\label{lem2.4}
If the mean curvature vector $\vec{H}$ of $\Sigma$ is  spacelike, and $\{e_3,e_4\}$ is any frame for the normal bundle of $\Sigma$ with the properties that $e_3$ is spacelike and $\langle\vec{H},e_3\rangle>0$, then
\begin{equation}
\int_{\Sigma}\mathcal{H}(e_3,u_a)dA\geq \int_{\Sigma}\mathcal{H}(\bar{e}_3,u_a) dA
\end{equation}
for all $\varepsilon>0$.
\end{lemma}

\begin{proof}
Consider the following functional which sends normal bundle frames to the real numbers
\begin{equation}\label{framemap}
\{e_3,e_4\} \mapsto\int_{\Sigma}\mathcal{H}(e_3,u_a)dA
=\int_{\Sigma}\left(\sqrt{|\nabla_\partial u_a|^2+\varepsilon^2}\,\langle\vec{H},e_3\rangle
+\alpha_{e_3}(\nabla_\partial u_a)\right)dA.
\end{equation}
Since the normal bundle of $\Sigma$ is rank 2 with structure group $SO(1,1)$, each frame may be given by a hyperbolic angle function. In particular, because $\vec{H}$ is  spacelike and $\langle\vec{H},e_3\rangle>0$, we may express the frame defined by the mean curvature vector as
\begin{equation}\label{eq1}
\tilde{e}_3:=\frac{\vec{H}}{|\vec{H}|}=(\cosh f) e_3+(\sinh f) e_4,\qquad \tilde{e}_4 :=(\sinh f) e_3+(\cosh f) e_4,
\end{equation}
for some $f\in C^{\infty}(\Sigma)$. It follows that
\begin{equation}\label{aoihfoinahg}
\langle\vec{H},{e}_3\rangle=|\vec{H}|\cosh f,\qquad \langle\vec{H},{e}_4\rangle=-|\vec{H}|\sinh f.
\end{equation} 
Furthermore, using equations \eqref{2} and \eqref{eq1} produces
\begin{equation}
\alpha_{e_3}(\nabla_{\partial}u_a)=\alpha_{\tilde{e}_3}(\nabla_{\partial}u_a)+\nabla_{\partial}u_a\cdot\nabla_{\partial} f.
\end{equation}
Therefore, the functional of \eqref{framemap} can be rewritten as
\begin{align}\label{h1}
\begin{split}	f \mapsto &\int_{\Sigma}\left(\sqrt{|\nabla_{\partial}u_a|^2+\varepsilon^2}|\vec{H}|\cosh f+\nabla_{\partial}u_a\cdot\nabla_{\partial} f+\alpha_{\tilde{e}_3}(\nabla_{\partial}u_a)\right) dA\\
=&\int_{\Sigma}\left(\sqrt{|\nabla_{\partial}u_a|^2+\varepsilon^2}|\vec{H}|\cosh f-f\Delta_{\partial}u_a+\alpha_{\tilde{e}_3}(\nabla_{\partial}u_a)\right) dA.
\end{split}
\end{align}
Since $|\vec{H}|>0$ this functional is convex, and it may be easily checked that the minimum occurs when 
\begin{equation}\label{theta1}
|\vec{H}|\sinh f=\frac{\Delta_{\partial}u_a}{\sqrt{|\nabla_{\partial}u_a|^2+\varepsilon^2}}.
\end{equation}
Hence, \eqref{aoihfoinahg} shows that the minimum is achieved at the frame $\{\bar{e}_3,\bar{e}_4\}$. 
\end{proof}

Let $(\Omega, g,k)$ be the initial data set for a compact spacelike hypersurface in $N^{3,1}$ which is enclosed by $\Sigma$, and let $u_a \in C^{\infty}(\Sigma)$. Consider the unique solution of the spacetime harmonic Dirichlet problem
\begin{equation}\label{q1}
\Delta_g u+(\text{Tr}_{g}k)|\nabla u|=0\quad \text{ in }\Omega,\qquad u=u_a\quad \text{ on }\partial\Omega=\Sigma.
\end{equation}
The existence of a unique solution $u\in C^{2,\varsigma}(\Omega)$ for any $\varsigma\in(0,1)$, follows in a straightforward manner from the results of \cite[Section 4.1]{HKK}. We may then define a level set frame for the normal bundle of $\Sigma$ by
\begin{equation}\label{primeframe}
{e}'_3=\frac{\sqrt{|\nabla u|^2+\varepsilon^2}}{\sqrt{|\nabla_{\partial}u_a|^2+\varepsilon^2}}\nu + \frac{\nu(u)}{\sqrt{|\nabla_{\partial}u_a|^2+\varepsilon^2}}\mathbf{n},\qquad {e}'_4= \frac{\nu(u)}{\sqrt{|\nabla_{\partial}u_a|^2+\varepsilon^2}}\nu +\frac{\sqrt{|\nabla {u}|^2+\varepsilon^2}}{\sqrt{|\nabla_{\partial}u_a|^2+\varepsilon^2}}\mathbf{n},
\end{equation}
for each $\varepsilon>0$ and where $\{\nu,\mathbf{n}\}$ is the normal bundle frame determined by $\Omega$ in which $\nu$ is the outer normal to $\partial\Omega$ and $\mathbf{n}$ is future directed timelike.
Note that since the mean curvature vector of $\Sigma$ is outward pointing spacelike, we have 
\begin{equation}
H=\langle \vec{H},\nu\rangle>|\langle\vec{H},\mathbf{n}\rangle|=|\mathrm{Tr}_{\Sigma}k|
\end{equation}
and therefore
\begin{align}
\begin{split}
\langle\vec{H},{e}'_3\rangle=&\frac{\sqrt{|\nabla u|^2+\varepsilon^2}}{\sqrt{|\nabla_{\partial} \hat{u}_a|^2+\varepsilon^2}}H +\frac{\nu(u)}{\sqrt{|\nabla_{\partial} \hat{u}_a|^2+\varepsilon^2}}\text{Tr}_{\Sigma}k\\
>& \frac{\sqrt{|\nabla u|^2+\varepsilon^2}}{\sqrt{|\nabla_{\partial} \hat{u}_a|^2+\varepsilon^2}} |\text{Tr}_{\Sigma}k| +\frac{\nu(u)}{\sqrt{|\nabla_{\partial} \hat{u}_a|^2+\varepsilon^2}}\text{Tr}_{\Sigma}k\\
\geq &\frac{|\nu(u)|}{\sqrt{|\nabla_{\partial} \hat{u}_a|^2+\varepsilon^2}} |\text{Tr}_{\Sigma}k| +\frac{\nu(u)}{\sqrt{|\nabla_{\partial} \hat{u}_a|^2+\varepsilon^2}}\text{Tr}_{\Sigma}k\\
\geq & 0.
\end{split}
\end{align}
This allows for an application of Lemma \ref{lem2.4} to conclude that
\begin{equation}\label{aoihfjoqihoihnq}
\int_{\Sigma}\mathcal{H}({e}'_3,u_a)\, dA\geq \int_{\Sigma}\mathcal{H}(\bar{e}_3,u_a)\, dA.
\end{equation}
Inequality \eqref{aoihfjoqihoihnq}, together with the following estimate, will be used together to show nonnegativity of the quasi-local energy.

\begin{lemma}\label{lemma2.3}
Let $(\Omega, g,k)$ be initial data for a compact spacelike hypersurface with boundary $\partial\Omega=\Sigma$ in spacetime $N^{3,1}$, and let $u_a \in C^{\infty}(\Sigma)$ and $\varepsilon>0$. If $u\in C^{2,\varsigma}(\Omega)$ is the spacetime harmonic solution of \eqref{q1} with associated level set frame $\{e'_3 ,e'_4\}$, then
\begin{equation}\label{mainin}
\lim_{\varepsilon\to 0}\int_{\Sigma}\mathcal{H}(e'_3,u_a) dA+\int_{\Omega}\left(\frac{1}{2}\frac{|{\nabla}^2 u+k|\nabla u||^2}{|\nabla u|}+\mu|\nabla u|+J(\nabla u)\right) dV\leq 2\pi \int_{\underline{u}}^{\overline{u}}\chi(\Sigma_s)ds,
\end{equation}
where $\Sigma_s=u^{-1}(s)$ and $\overline{u}$, $\underline{u}$ represent the maximum and minimum values of $u$. Furthermore, the limit in this expression exists and is finite. 
\end{lemma}

\begin{proof}
Let $f_{\varepsilon}\in C^{\infty}(\Sigma)$ be such that
\begin{equation}
\cosh f_{\varepsilon}=\frac{\sqrt{|\nabla u|^2+\varepsilon^2}}{\sqrt{|\nabla_{\partial}u_a|^2+\varepsilon^2}},\qquad 
\sinh f_{\varepsilon}=\frac{\nu(u)}{\sqrt{|\nabla_{\partial}u_a|^2+\varepsilon^2}}, 
\end{equation}
then \eqref{1} and \eqref{2} imply
\begin{equation}
\alpha_{e'_3}(\nabla_{\partial}u_a)=\alpha_{\nu}(\nabla_{\partial}u_a)-\nabla_{\partial}u_a\cdot\nabla_{\partial} f_{\varepsilon}
=-k(\nabla_{\partial}u_a ,\nu)-\nabla_{\partial}u_a\cdot\nabla_{\partial} f_{\varepsilon}.
\end{equation}
It follows that
\begin{align}
\begin{split}
\mathcal{H}(e'_3,u_a)
=&\sqrt{|\nabla_{\partial}u_a|^2+\varepsilon^2}\langle\vec{H},{e}'_3\rangle
+\alpha_{e'_3}(\nabla_{\partial}u_a)\\
=&\sqrt{|\nabla u|^2+\varepsilon^2}H +\nu(u)\text{Tr}_{\Sigma}k-k(\nabla_{\partial}u_a,\nu)-\nabla_{\partial}u_a\left(\sinh^{-1}\frac{\nu(u)}{\sqrt{|\nabla_{\partial}u_a|^2+\varepsilon^2}}\right)\\
=&\sqrt{|\nabla u|^2 \!+\!\varepsilon^2}H \!+\!\nu(u)\text{Tr}_{\Sigma}k \!-\! k(\nabla_{\partial}u_a,\nu) \!-\!\frac{\sqrt{|\nabla_{\partial}u_a|^2 \!+\!\varepsilon^2}}{\sqrt{|\nabla{u}|^2 \!+\!\varepsilon^2}}\nabla_{\partial}u_a\left(\!\frac{\nu(u)}{\sqrt{|\nabla_{\partial}u_a|^2 \!+\!\varepsilon^2}}\!\right).
\end{split}
\end{align} 
Since the last term in this expression may be estimated by
\begin{equation}
\frac{\sqrt{|\nabla_{\partial}u_a|^2 +\varepsilon^2}}{\sqrt{|\nabla{u}|^2 +\varepsilon^2}}
\Bigg|\frac{\nabla_{\partial}u_a\cdot\nabla_{\partial}\left(\nu(u)\right)}{\sqrt{|\nabla_{\partial}u_a|^2 +\varepsilon^2}}
-\frac{\nu(u)\nabla_{\partial}^2u_a(\nabla_{\partial}u_a,\nabla_{\partial}u_a)}{\left(|\nabla_{\partial}u_a|^2+\varepsilon^2\right)^{3/2}}\Bigg|
\leq|\nabla^2 u|+|II||\nabla_{\partial}u_a|+|\nabla_{\partial}^2 u_a|,
\end{equation}
where $II$ is the second fundamental form of $\Sigma$ as a submanifold of $(\Omega,g)$, we may apply the dominated convergence theorem to conclude that the relevant limit exists, is finite, and satisfies
\begin{align}\label{aohfoaihoghiqphj}
\begin{split}
\lim_{\varepsilon\to 0}\int_{\Sigma}\mathcal{H}(e'_3,u_a) dA
=&\int_{\partial\Omega}\left(|\nabla u|H+\nu(u)\text{Tr}_\Sigma k -k(\nabla_{\partial} u_a,\nu)\right)dA\\
&-\int_{\bar{\partial}\Omega}\frac{|\nabla_{\partial}u_a|}{|\nabla u|}\nabla_{\partial}u_a\left(\frac{\nu(u)}{|\nabla_{\partial}u_a|}\right)dA.       
\end{split}
\end{align}
Here $\bar{\partial}\Omega$ denotes the open set of points within $\partial\Omega$ on which $|\nabla_{\partial}u_a|\neq 0$.
The desired result now follows from Proposition \ref{integralidentity}.
\end{proof}

\begin{proof}[Proof of Theorem \ref{thm1}: Nonnegativity]
A direct application of \eqref{aoihfjoqihoihnq}, Lemma \ref{lemma2.3}, and Lemma \ref{caseofeq} (in the next section)
yields
\begin{align}\label{apooihgoiqhoihj}
\begin{split}
E(\Sigma,\iota,u_a)
=&\lim_{\varepsilon\rightarrow 0}\frac{1}{8\pi}\int_{\Sigma}\left(\mathcal{H}_0(\hat{e}_3,u_a)-\mathcal{H}(\bar{e}_3,u_a)\right)dA\\
\geq & \lim_{\varepsilon\rightarrow 0}\frac{1}{8\pi}\int_{\Sigma}\left(\mathcal{H}_0(\hat{e}_3,u_a)-
\mathcal{H}({e}'_3,u_a)\right)dA\\
\geq & \int_{\Omega}\left(\frac{1}{2}\frac{|{\nabla}^2 u+k|\nabla u||^2}{|\nabla u|}+\mu|\nabla u|+J(\nabla u)\right) dV\\
&+\frac{1}{4} \int_{\underline{u}_a}^{\overline{u}_a}\chi(\hat{\Sigma}_s)ds -\frac{1}{4} \int_{\underline{u}}^{\overline{u}}\chi(\Sigma_s)ds,
\end{split}
\end{align}
where $\hat{\Sigma}_s$ are the level sets of the relevant null linear function in Minkowski space restricted to the fill-in $\hat{\Omega}$ of $\iota(\Sigma)$. By the maximum principle for \eqref{q1} we find that $\overline{u}=\overline{u}_a$ and $\underline{u}=\underline{u}_a$. Furthermore, analogous arguments to those used in \cite[Proposition 5.2]{HKK} show that the trivial homology hypothesis $H_2(\Omega;\mathbb{Z})=0$ guarantees that each component of any regular level $\Sigma_s$ for $u$ must intersect the boundary $\partial\Omega$. Hence, $\chi(\Sigma_s)\leq n$ where $n$ is the number of components of the $s$-level set for $u_a=u|_{\partial\Omega}$. On the other hand, the admissibility condition ensures that $\chi(\hat{\Sigma}_s)=n$, so the difference of Euler characteristic integrals in \eqref{apooihgoiqhoihj} is nonnegative.
The dominant energy condition $\mu\geq|J|$ then gives $E(\Sigma,\iota,u_a)\geq 0$.

It remains to show that the limit defining the quasi-local energy exists and is finite. From the proof of Lemma \ref{lem2.4} it follows that
\begin{align}\label{oqignoihihniqh}
\begin{split}
\mathcal{H}(\bar{e}_3,u_a)=&\sqrt{|\vec{H}|^2 (|\nabla_{\partial} u_a|^2 +\varepsilon^2)+(\Delta_{\partial}u_a)^2}+\alpha_{\tilde{e}_3}(\nabla_{\partial}u_a)\\
&+\nabla_{\partial}u_a \cdot \nabla_{\partial}\left(\sinh^{-1}\frac{\Delta_{\partial}u_a}{|\vec{H}|\sqrt{|\nabla_{\partial} u_a|^2 +\varepsilon^2}}\right) .
\end{split}
\end{align}
The last term in this expression may be estimated by
\begin{align}\label{oiqhoihgoqih}
\begin{split}
&\Bigg|\frac{\nabla_{\partial}u_a (\Delta_{\partial}u_a)-(\Delta_{\partial}u_a)\nabla_{\partial}u_a(\log|\vec{H}|)\!-\!(\Delta_{\partial}u_a)\nabla^2_{\partial}u_a(\nabla_{\partial}u_a,\!\nabla_{\partial}u_a)(|\nabla_{\partial} u_a|^2 \!+\!\varepsilon^2)^{-1}}
{\sqrt{|\vec{H}|^2 (|\nabla_{\partial} u_a|^2 \!+\!\varepsilon^2)\!+\!(\Delta_{\partial}u_a)^2}}\Bigg|\\
\leq& |\vec{H}|^{-1}|\nabla_{\partial}^3 u_a|+|\nabla_{\partial}u_a||\nabla_{\partial}\log|\vec{H}||+|\nabla_{\partial}^2 u_a|.
\end{split}
\end{align}
We may then apply the dominated convergence theorem to find that the limit exists, is finite, and satisfies
\begin{align}\label{ofnqoignoiwnqhoinq}
\begin{split}
\lim_{\varepsilon\rightarrow 0}\int_{\Sigma}\mathcal{H}(\bar{e}_3,u_a)dA
=&\int_{\Sigma}\left(\sqrt{|\vec{H}|^2 |\nabla_{\partial} u_a|^2 \!+\!(\Delta_{\partial}u_a)^2}+\alpha_{\tilde{e}_3}(\nabla_{\partial}u_a)\right)dA\\
&+\int_{\tilde{\Sigma}}\nabla_{\partial}u_a \cdot \nabla_{\partial}\left(\sinh^{-1}\frac{\Delta_{\partial}u_a}{|\vec{H}||\nabla_{\partial} u_a|}\right)dA,
\end{split}
\end{align}
where $\tilde{\Sigma}$ denotes the open subset of $\Sigma$ on which $|\nabla_{\partial}u_a|\neq 0$. Similar arguments hold for the limit of the reference Hamiltonian.
\end{proof}

\begin{proof}[Nonnegativity with a disconnected surface]
In Remark \ref{remark1}, it was stated that nonnegativity of the energy holds in the more general circumstance of a disconnected $\Sigma$ and without the homology assumption on $\Omega$, when a suitable modification of the admissibility condition is enforced. More precisely, in this setting we may define a pair $(\iota,u_a)$ to be \textit{admissible} for $\Sigma\subset N^{3,1}$ if the mean curvature vector of $\iota(\Sigma)$ is outward pointing spacelike, and there exist compact spacelike hypersurfaces $\hat{\Omega}\subset\mathbb{R}^{3,1}$, $\Omega\subset N^{3,1}$ with $\partial\hat{\Omega}=\iota(\Sigma)$, $\partial\Omega=\Sigma$ such that 
\begin{equation}
\int_{\underline{u}_a}^{\overline{u}_a}\left(\chi(\hat{\Sigma}_s)-\chi(\Sigma_s)\right)ds \geq 0,
\end{equation}
where the level sets $\hat{\Sigma}_s$, $\Sigma_s$ are defined as above. Since \eqref{aoihfjoqihoihnq}, Lemma \ref{lemma2.3}, and Lemma \ref{caseofeq} continue to hold under the more general hypotheses presented here, inequality \eqref{apooihgoiqhoihj} again implies that $E(\Sigma,\iota,u_a)\geq 0$.
\end{proof}

\section{Proof of Rigidity}
\label{sec5} \setcounter{equation}{0}
\setcounter{section}{5}

The purpose of the current section is to establish the rigidity statement of Theorem \ref{thm1}. We will first compute the surface Hamiltonian for surfaces in Minkowski space. Let $\hat{\Sigma}=\iota(\Sigma)\subset\mathbb{R}^{3,1}$ be a spacelike 2-surface, and assume that $(\iota,u_a)$ is admissible.
There is then a compact spacelike hypersurface $(\hat{\Omega},\hat{g},\hat{k})$ in Minkowski space with boundary $\partial\hat{\Omega}=\hat{\Sigma}$. Let $\{\hat{\nu},\hat{\mathbf{n}}\}$ be a normal bundle frame for $\hat{\Sigma}$, where $\hat{\nu}$ is the outer normal with respect to $\hat{\Omega}$ and $\hat{\mathbf{n}}$ is future directed timelike. For $\varepsilon>0$ consider the level set frame
\begin{equation}\label{checke}
\hat{e}'_3=\frac{\sqrt{|\nabla \hat{u}|^2+\varepsilon^2}}{\sqrt{|\nabla_{\partial}u_a|^2+\varepsilon^2}}\hat{\nu} + \frac{\hat{\nu}(\hat{u})}{\sqrt{|\nabla_{\partial}u_a|^2+\varepsilon^2}}\hat{\mathbf{n}},\qquad \hat{e}'_4= \frac{\hat{\nu}(\hat{u})}{\sqrt{|\nabla_{\partial}u_a|^2+\varepsilon^2}}\hat{\nu} +\frac{\sqrt{|\nabla \hat{u}|^2+\varepsilon^2}}{\sqrt{|\nabla_{\partial}u_a|^2+\varepsilon^2}}\hat{\mathbf{n}},
\end{equation}
where $\hat{u}=(-\mathbf{t}+a_i \mathbf{x}^i)|_{\hat{\Omega}}$ is the restriction of the null linear function to the hypersurface, $u_a$ is the restriction of $\hat{u}$ to $\hat{\Sigma}$, with $\nabla$ and $\nabla_{\partial}$ denoting the connections on $\hat{\Omega}$ and $\hat{\Sigma}$ respectively. Note that since the spacetime gradient of the linear function is null it holds that $\hat{\mathbf{n}}(-\mathbf{t}+a_i \mathbf{x}^i)=-|\nabla\hat{u}|$. Moreover, linearity of this function implies that
\begin{align}\label{shmin}
\begin{split}
0=\hat\square(-\mathbf{t}+a_i \mathbf{x}^i)=& \left(\hat{\nabla}_{\hat{\nu}\hat{\nu}}+\hat{\nabla}_{\hat{\mathbf{n}}\hat{\mathbf{n}}}+\vec{H}_0+\Delta_{\partial}\right)(-\mathbf{t}+a_i \mathbf{x}^i)\\
=&\left(\hat{H}\hat{\nu}-(\mathrm{Tr}_{\hat{\Sigma}}\hat{k})\hat{\mathbf{n}}
+\Delta_{\partial}\right)(-\mathbf{t}+a_i \mathbf{x}^i)\\
=&\hat{H}\hat{\nu}(\hat{u})+(\mathrm{Tr}_{\hat{\Sigma}}\hat{k})|\nabla\hat{u}|
+\Delta_{\partial}u_a,
\end{split}
\end{align}
where $\hat\square$ and $\hat{\nabla}$ represent the wave operator and connection on Minkowski space, and $\vec{H}_0=\hat{H}\hat{\nu}-(\mathrm{Tr}_{\hat{\Sigma}}\hat{k})\hat{\mathbf{n}}$ is the mean curvature vector of $\hat{\Sigma}$.
Since $\hat{\mathbf{n}}$ is future pointing timelike, the null condition for the linear function also gives $|\nabla\hat{u}|>0$ on $\hat{\Omega}$. Therefore, an examination of the proof for Proposition \ref{integralidentity} shows that the inequality of \eqref{ident} is in fact an equality. This and \eqref{aohfoaihoghiqphj}, combined with the observation that $\hat{u}$ has vanishing spacetime Hessian \cite[Section 5]{BHKKZ1} (see also \cite[Section 3]{HKK}), and using that $\mu=|J|=0$ in Minkowski space, yield a computation of the surface Hamiltonian
\begin{align}\label{H_0euler}
\begin{split}
\lim_{\varepsilon\to 0}\int_{\hat{\Sigma}}\mathcal{H}_0(\hat{e}'_3,u_a) d\hat{A}
=&\int_{\hat{\Sigma}}\left(|\nabla \hat{u}|\hat{H}+\hat{\nu}(\hat{u})\text{Tr}_{\hat{\Sigma}} \hat{k} -\hat{k}(\nabla_{\partial} u_a,\hat{\nu})\right)d\hat{A}\\
&-\int_{\bar{\partial}\hat{\Omega}}\frac{|\nabla_{\partial}u_a|}{|\nabla \hat{u}|}\nabla_{\partial}u_a\left(\frac{\hat{\nu}(\hat{u})}{|\nabla_{\partial}u_a|}\right)d\hat{A}\\       
=&2\pi \int_{\underline{u}_a}^{\overline{u}_a}\chi(\hat{\Sigma}_s)ds,
\end{split}
\end{align}
where $\hat{\Sigma}_s$ and $\overline{u}_a$, $\underline{u}_a$ denote the level sets, maximum, and minimum of $\hat{u}$ respectively. The next result shows that the same value is achieved by evaluating at the optimal frame, which is uniquely determined by
\begin{equation}
\langle\vec{H}_0,\hat{e}_3\rangle>0,\qquad \quad \langle\vec{H}_0,\hat{e}_4\rangle=\frac{-\Delta_\partial u_a}
{\sqrt{|\nabla_\partial u_a|^2+\varepsilon^2}}.
\end{equation}

\begin{lemma}\label{caseofeq}
Let $(\iota,u_a)$ be an admissible pair for a spacelike 2-surface $\Sigma$, then the reference Hamiltonian satisfies
\begin{equation}
\lim_{\varepsilon\to 0}\int_{\Sigma}\mathcal{H}_0(\hat{e}_3,u_a)dA=2\pi \int_{\underline{u}_a}^{\overline{u}_a}\chi(\hat{\Sigma}_s)ds.
\end{equation}
\end{lemma}

\begin{proof}
For convenience we will remove extraneous notation, including the subscript on the mean curvature vector and the hat notation from all objects except the optimal frame. Consider the frame associated with $\Omega$, namely
\begin{equation}
\tilde{e}_3=\frac{\vec{H}}{|\vec{H}|}=\frac{H}{|\vec{H}|}\nu -\frac{(\mathrm{Tr}_{\Sigma}k)}{|\vec{H}|}\mathbf{n}
:=(\cosh\ell) \nu-(\sinh\ell) \mathbf{n},\quad\quad\quad
\tilde{e}_4=-(\sinh\ell) \nu+(\cosh\ell)\mathbf{n}.
\end{equation}
This may be used as an intermediary frame to compute the relation between optimal frame $\{\hat{e}_3,\hat{e}_4\}$ and the $\Omega$ frame $\{\nu,\mathbf{n}\}$. In particular, since
\begin{equation}
\hat{e}_3=(\cosh f_{\varepsilon})\tilde{e}_3 -(\sinh f_{\varepsilon})\tilde{e}_4,\quad\quad\quad \hat{e}_4=-(\sinh f_{\varepsilon})\tilde{e}_3
+(\cosh f_{\varepsilon})\tilde{e}_4
\end{equation}
where
\begin{equation}
\cosh f_{\varepsilon}=\frac{\sqrt{|\vec{H}|^2 (|\nabla_{\partial}u_a|^2+\varepsilon^2)+(\Delta_{\partial}u_a)^2}}{|\vec{H}|\sqrt{|\nabla_{\partial}u_a|^2+\varepsilon^2}},\quad\quad
\sinh f_{\varepsilon}=\frac{\Delta_{\partial}u_a}{|\vec{H}|\sqrt{|\nabla_{\partial}u_a|^2+\varepsilon^2}},
\end{equation}
it follows that
\begin{equation}
\hat{e}_3=(\cosh q_{\varepsilon})\nu-(\sinh q_{\varepsilon})\mathbf{n},\quad\quad\quad
\hat{e}_4 =-(\sinh q_{\varepsilon})\nu+(\cosh q_{\varepsilon})\mathbf{n}
\end{equation}
with $q_{\varepsilon}=f_{\varepsilon}+\ell$ and
\begin{align}\label{qoihoihqogiq}
\begin{split}
\cosh q_{\varepsilon} =& \frac{\sqrt{|\vec{H}|^2 (|\nabla_{\partial}u_a|^2+\varepsilon^2)+(\Delta_{\partial}u_a)^2}}{|\vec{H}|\sqrt{|\nabla_{\partial}u_a|^2+\varepsilon^2}}\cdot\frac{H}{|\vec{H}|}
+\frac{\Delta_{\partial}u_a}{|\vec{H}|\sqrt{|\nabla_{\partial}u_a|^2+\varepsilon^2}}\cdot\frac{\mathrm{Tr}_{\Sigma}k}{|\vec{H}|}\\
\sinh q_{\varepsilon} =&-\frac{\Delta_{\partial}u_a}{|\vec{H}|\sqrt{|\nabla_{\partial}u_a|^2+\varepsilon^2}}\cdot\frac{H}{|\vec{H}|}
-\frac{\sqrt{|\vec{H}|^2 (|\nabla_{\partial}u_a|^2+\varepsilon^2)+(\Delta_{\partial}u_a)^2}}{|\vec{H}|\sqrt{|\nabla_{\partial}u_a|^2+\varepsilon^2}}\cdot\frac{\mathrm{Tr}_{\Sigma}k}{|\vec{H}|}.
\end{split}
\end{align}

The Hamiltonian density function may then be rewritten with respect to the $\Omega$ frame utilizing \eqref{1} and \eqref{2} to obtain
\begin{align}\label{oqh8h888}
\begin{split}
\mathcal{H}_0(\hat{e}_3,u_a)=&\sqrt{|\nabla_{\partial}u_a|^2 +\varepsilon^2}\langle\vec{H},\hat{e}_3\rangle+\alpha_{\hat{e}_3}(\nabla_{\partial}u_a)\\
=&\sqrt{|\nabla_{\partial}u_a|^2 +\varepsilon^2}\left(H\cosh q_{\varepsilon} +(\mathrm{Tr}_{\Sigma}k)\sinh q_{\varepsilon}\right)-k(\nabla_{\partial}u_a ,\nu)-\nabla_{\partial}u_a \cdot\nabla_{\partial} q_{\varepsilon}.
\end{split}
\end{align}
By employing the fact that the spacetime Hessian of $u$ vanishes on $\Omega$, or rather solving for the Laplacian in \eqref{shmin}, we find that
\begin{equation}
\Delta_{\partial}u_a=-H\nu(u)-(\mathrm{Tr}_{\Sigma}k)|\nabla u|.
\end{equation}
This expression may then be inserted into \eqref{qoihoihqogiq} to produce
\begin{equation}\label{rightarrow}
\sqrt{|\nabla_{\partial}u_a|^2 +\varepsilon^2}H\cosh q_{\varepsilon} \rightarrow H|\nabla u|,\quad\quad\quad
\sqrt{|\nabla_{\partial}u_a|^2 +\varepsilon^2}(\mathrm{Tr}_{\Sigma}k)\sinh q_{\varepsilon} \rightarrow (\mathrm{Tr}_{\Sigma}k)\nu(u),
\end{equation}
as $\varepsilon\rightarrow 0$. Moreover, similarly to \eqref{oqignoihihniqh} and \eqref{oiqhoihgoqih} we can estimate
\begin{align}
\begin{split}
|\nabla_{\partial}u_a \cdot\nabla_{\partial} q_{\varepsilon}|\leq & 2\left(\frac{|H|+|\mathrm{Tr}_{\Sigma}k|}{|\vec{H}|}\right)(|\nabla_{\partial}^3 u_a|+|\nabla_{\partial}^2u_a|+|\nabla_{\partial}\log|\vec{H}|||\nabla_{\partial}u_a|)\\
&+2\left(|\nabla_{\partial}(|\vec{H}|^{-1}H)|+|\nabla_{\partial}(|\vec{H}|^{-1}\mathrm{Tr}_{\Sigma}k)|\right)|\nabla_{\partial} u_a|.
\end{split}
\end{align}
Therefore the dominated convergence theorem may be employed to yield
\begin{equation}\label{jhqijgoiqajgpjqpahgoj}
\lim_{\varepsilon\rightarrow 0}\int_{\Sigma}\left(\nabla_{\partial}u_a \cdot\nabla_{\partial} q_{\varepsilon} \right)dA
=\int_{\bar{\partial}\Omega}\frac{|\nabla_{\partial}u_a|}{|\nabla u|}\nabla_{\partial}u_a\left(\frac{\nu(u)}{|\nabla_{\partial}u_a|}\right)dA.
\end{equation}
The desired result now follows from the second equality in \eqref{H_0euler}, together with \eqref{oqh8h888}, \eqref{rightarrow}, and \eqref{jhqijgoiqajgpjqpahgoj}.
\end{proof}

\begin{proof}[Proof of Theorem \ref{thm1}: Rigidity]
To establish part \textit{(3)} of this of this theorem, it suffices to observe that if the inclusion map of a spacelike 2-surface in Minkowski space is admissible with some $u_a$, then the physical and reference Hamiltonian densities are the same, so that the associated energy is zero. It then follows from nonnegativity in part \textit{(1)} that the infimum of energies, and hence the mass of this surface, is zero. 

Consider now part \textit{(2)}. Suppose that for some spacelike 2-surface the pair $(\iota,u_a)$ is admissible, and $E(\Sigma,\iota,u_a)=0$ for all $\mathbf{a}$. Let $(\Omega,g,k)$ be a compact initial data set satisfying the dominant energy condition, with trivial second homology, and boundary $\partial\Omega=\Sigma$. The spacetime harmonic function on $\Omega$ with boundary values $u_a$, will also be denoted by $u_a$. We will follow the general strategy of \cite[Section 7]{HKK}. First observe that for each $\mathbf{a}$ the inequality \eqref{apooihgoiqhoihj} implies 
\begin{equation}\label{oihoiqhohnoqing}
{\nabla}^2 u_a+k|\nabla u_a|=0,\quad\quad\quad \mu=|J|=0   \quad\quad\text{ on }\Omega,
\end{equation}
whenever $|\nabla u_a|\neq 0$. Moreover, the Hopf lemma applied to a maximum point $x_0 \in \Sigma$ of $u_a$ shows that $|\nabla u_a (x_0)|\neq 0$.  Let $\gamma\subset\Omega$ be a curve emanating from $x_0$ parameterized by arclength, and observe that since
\begin{equation}
|\nabla|\nabla u_a||\leq|\nabla^2 u_a|\leq |k||\nabla u_a|
\end{equation}
holds away from critical points, we have 
\begin{equation}
|(\log|\nabla u_a|\circ\gamma)'|\leq |\nabla\log|\nabla u_a||\circ\gamma\leq C
\end{equation}
away from critical points, for some constant $C$. By integrating along arbitrary $\gamma$, it follows that there is a constant $C_1$ such that
\begin{equation}
C_1^{-1}|\nabla u_a(x_0)|\leq|\nabla u_a(x)|\leq C_1 |\nabla u_a(x_0)|,\quad\quad\quad x\in\Omega.
\end{equation}
Hence, $|\nabla u_a|$ does not vanish globally for any $\mathbf{a}$. 

We will now choose three spacetime harmonic functions $u_1$, $u_2$, and $u_3$ on $\Omega$ in the following way.
Let $u_3$ be the spacetime harmonic function associated with $\mathbf{a}_3=(0,0,1)$, and consider a point $p\in \Sigma$ at which $u_3$ achieves its maximum; at this point $|\nabla_{\partial}u_3(p)|=0$. In Minkowski space, the isometric image $\iota(\Sigma)$ is tangent at $\iota(p)$ to a hyperplane $-\mathbf{t}+\mathbf{x}^3 =const$ and lies to one side of it. By computing the projection of Minkowski space gradients for the functions $-\mathbf{t}+a_i\mathbf{x}^i$ onto the tangent space $T_{\iota(p)}\iota(\Sigma)$, it is possible to find two choices for $\mathbf{a}$ such that the corresponding spacetime harmonic functions $u_1$, $u_2$ have nonvanishing boundary gradients at $p$ and satisfy $\nabla_{\partial}u_1(p) \perp \nabla_{\partial}u_2(p)$. It follows that the vector fields $\nabla u_l$, $l=1,2,3$ are linearly independent on $\Omega$.

Consider the quantities
\begin{equation}
\varphi=|\nabla u_1|+|\nabla u_2|+|\nabla u_3|, \quad\quad\quad
Y=\nabla u_1 +\nabla u_2 +\nabla u_3,
\end{equation}
and build the stationary spacetime $(\mathbb{R}\times\Omega,\bar{g})$ where
\begin{equation}
\bar{g}=-(\varphi^2 -|Y|^2) dt^2+2Y_i dx^i dt +g.
\end{equation}
This is the Kiling development of $(\Omega,g,k,\varphi,Y)$ with Killing initial data lapse-shift $(\varphi,Y)$ decomposing the Killing vector $\partial_t =\varphi \mathbf{n}+Y$, where $\mathbf{n}$ is the future pointing unit normal to the constant time slices.
As is shown in \cite[proof of Theorem 7.3]{HKK} the initial data for these slices is $(\Omega,g,k)$, and as a consequence of the vanishing spacetime Hessians \eqref{oihoiqhohnoqing} the function $\varphi^2-|Y|^2 =:c^2$ is constant. Furthermore, observe that 
\begin{equation}
\varphi^2-|Y|^2 =2\sum_{l<m}\left(|\nabla u_l||\nabla u_m|-\nabla u_l \cdot\nabla u_m\right)=\sum_{l<m}|\nabla u_l||\nabla u_m|\left|\frac{\nabla u_l}{|\nabla u_l|}-\frac{\nabla u_m}{|\nabla u_m|}\right|^2.
\end{equation}
Since $|\nabla u_l|$ never vanishes, if $c=0$ then $\nabla u_l \parallel \nabla u_m$ for all $l,m$. In particular, there exist constants $c_{lm}>0$ such that at a given point $x_1\in\Omega$ it holds that
\begin{equation}\label{qiupojqgpoijqg}
\nabla u_l(x_1)-c_{lm} \nabla u_m(x_1)=0.
\end{equation}
We claim that these relations hold at all points. To see this, note that \eqref{oihoiqhohnoqing} implies
\begin{equation}
|\nabla|\nabla (u_l-c_{lm} u_m)||\leq|\nabla^2(u_l -c_{lm} u_m)|\leq |k|||\nabla u_l|-c_{lm}|\nabla u_m||
\leq|k||\nabla(u_l -c_{lm}u_m)|.
\end{equation}
Then integrating along curves emanating from $x_1$ produces
\begin{equation}
C_2^{-1}|\nabla(u_l -c_{lm}u_m)(x_1)|\leq |\nabla(u_l -c_{lm}u_m)(x)|\leq C_2 |\nabla(u_l -c_{lm}u_m)(x_1)|
\end{equation}
for some constant $C_2>0$ and all $x\in\Omega$, yielding the desired claim. However, \eqref{qiupojqgpoijqg} cannot hold at $p\in\Sigma$ due to the properties of the boundary gradients $\nabla_{\partial}u_l(p)$. We conclude that the constant $c\neq 0$, and hence 
\begin{equation}
\bar{g}=-\left(cdt-c^{-1}Y_i dx^i \right)^2+(g_{ij}+c^{-2}Y_i Y_j )dx^i dx^j=-d\bar{t}^2 +(g+d\mathbf{u}^2)
\end{equation}
where $\bar{t}=ct-c^{-1}\mathbf{u}$ and $\mathbf{u}=u_1+u_2+u_3$.

We will now show that the Killing development is isometric to a portion of Minkowski space. Consider the null vector fields
\begin{equation}\label{ohtoiqhoighoqihg}
X_l =\nabla \tilde{u}_l +|\nabla \tilde{u}_l|\mathbf{n}, \quad\quad\quad l=1,2,3,
\end{equation}
where the functions $\tilde{u}_l$ are spacetime harmonic functions on $\Omega$ extended trivially in the $t$-direction to all of $\mathbb{R}\times\Omega$. These functions are chosen in the following way. Let $p_1 \in\Sigma$ be a maximum point for $u_1+u_2+u_3$ on $\Sigma$. Then in Minkowski space, the isometric image $\iota(\Sigma)$ is tangent at $\iota(p_1)$ to a hyperplane $-\mathbf{t}+b_i\mathbf{x}^i =const$ and lies to one side of it, where $|\mathbf{b}|\leq 1$. We may rotate this to a null hyperplane which is still tangent to $\iota(\Sigma)$ at $\iota(p_1)$, and which is defined by $-\mathbf{t}+\tilde{a}_i \mathbf{x}^i =const$ with $|\tilde{\mathbf{a}}|=1$. We then set $\tilde{u}_3$ to be the associated spacetime harmonic function, and note that $|\nabla_{\partial}\tilde{u}_3(p_1)|=0$. As above, we may also find two additional spacetime harmonic functions $\tilde{u}_1$, $\tilde{u}_2$ having nonvanishing boundary gradients at $p_1$ and satisfying $\nabla_{\partial}u_1(p_1) \perp \nabla_{\partial}u_2(p_1)$. Since
\begin{equation}
\sum_l d_l X_l + d_4 \partial_t =\sum_l d_l \nabla\tilde{u}_l + d_4 \nabla(u_1 +u_2 +u_3)
+\left(\sum_l d_l |\nabla\tilde{u}_l| +d_4(|\nabla u_1|+|\nabla u_2|+|\nabla u_3|)\right)\mathbf{n}
\end{equation}
for any constants $d_1,\ldots,d_4$, we find that setting this quantity to zero implies $d_1=d_2=0$ by evaluating on $T_{p_1}\Sigma$. Moreover, since $X_3$ is null and $\partial_t$ is timelike, it follows that $\{X_1,X_2,X_3,\partial_t\}$
is linearly independent. Additionally, note that from \eqref{oihoiqhohnoqing} the functions $\tilde{u}_l$ have vanishing spacetime Hessians, which as in \cite[proof of Theorem 7.3]{HKK} implies that each of these four vector fields is covariantly constant in spacetime. Hence $(\mathbb{R}\times \Omega,\bar{g})$ is flat, and consequently the Riemannian manifold $(\Omega,g+d\mathbf{u}^2)$ is flat. It remains to show that this manifold is isometric to a domain in Euclidean 3-space.

The vanishing quasi-local energy, the admissiblity condition, and inequality \eqref{apooihgoiqhoihj} show that the Euler characteristics agree $\chi(\Sigma_s)=\chi(\hat{\Sigma}_s)$ for regular values $s$. Here $\Sigma_s$ is the $s$-level set of an arbitrary spacetime harmonic function $u_a$, and $\hat{\Sigma}_s$ is the $s$-level set of the corresponding null linear function in $\hat{\Omega}\subset\mathbb{R}^{3,1}$ where $\hat{\Omega}$ is a spacelike fill-in for $\iota(\Omega)$. It follows that $\Omega$ is diffeomorphic to $\hat{\Omega}$.
Consider now the dual 1-forms to the vector fields $X_l$ on spacetime, and restrict them to the $\bar{t}=0$ slice to obtain 1-forms $\omega^l$, $l=1,2,3$ on $(\Omega,g+d\mathbf{u}^2)$. Since the $\omega^l$ are covariantly constant, they are closed. We claim that they are also exact. To see this, it will be shown that they integrate to zero on any closed curve $\mathbf{c}\subset \Omega$. Indeed, such a $\mathbf{c}$ is homologous to a closed curve $\tilde{\mathbf{c}}$ inside the $\bar{t}=-c^{-1}\mathbf{u}$ slice, which coincides with the initial data $(\Omega,g,k)$. Furthermore, according to \eqref{ohtoiqhoighoqihg} the restriction of the $X_l$ dual 1-forms to this hypersurface is exact, and hence integrates to zero along $\tilde{\mathbf{c}}$. Hence, there exist functions $\bar{u}^l \in C^{2,\varsigma}(\Omega)$, $l=1,2,3$ such that $\omega^l=d\bar{u}^l$. Since the $\omega^l$ are covariantly constant, by applying a Gram-Schmidt procedure we may assume that they are orthonormal, and in addition the Hessians of the $\bar{u}^l$ vanish. Therefore
\begin{equation}
g+d\mathbf{u}^2 =(d\bar{u}^1)^2 +(d\bar{u}^2)^2 +(d\bar{u}^3)^2,
\end{equation}
and $(\bar{u}_1,\bar{u}_2,\bar{u}_3)$ can be used as a system of global coordinates on the $\bar{t}=0$ slice.
It follows that the Killing development of $(\Omega,g,k)$ is isometric to a portion of Minkowski space.
\end{proof}

\section{Asymptotics of the Energy}
\label{sec6} \setcounter{equation}{0}
\setcounter{section}{6}

In this section we will establish Theorem \ref{thm2}, which shows that the quasi-local energy asymptotes to the appropriate ADM quantity along coordinate spheres in an asymptotically flat end. Let $(M,g,k)$ be an asymptotically flat initial data set for the Einstein equations, and consider a coordinate sphere $S_r \subset M$. According to \cite[Lemma 2.1]{FST} the mean and Gauss curvatures of these spheres satisfy the following expansions 
\begin{equation}
H=\frac{2}{r}+O(r^{-1-\tau}),\quad\quad\quad\quad K=\frac{1}{r^2}+O(r^{-2-\tau}),
\end{equation}
where $\tau>\tfrac{1}{2}$ is the asymptotic flatness parameter of \eqref{asymflat}. It follows that for sufficiently large $r$ the Gauss curvature is positive, and hence from \cite[pg. 353]{Nirenberg} (see also \cite[(2.18)]{FST}) there exists an isometric embedding into a constant time slice of Minkowski space $\iota_r :S_r \hookrightarrow\mathbb{R}^3 \subset \mathbb{R}^{3,1}$ such that 
\begin{equation}\label{oqhgoihqoihioqh}
|\nabla_{\partial}^l( \iota_r -\mathrm{id}_r)|=O(r^{1-\tau-l}),\quad\quad\quad l=0,1,2, 
\end{equation}
where $\mathrm{id}_r$ is the identity map on the sphere of radius $r$ in $\mathbb{R}^3$. It follows that the null linear function pullback, used to define the quasi-local energy, may be approximated by a linear combination of asymptotically flat coordinates. In particular, if $(x^1,x^2,x^3)$ are coordinates from \eqref{asymflat} in the asymptotic end of $M$ and $u_a=\iota_r^*(-\mathbf{t}+a_i \mathbf{x}^i)$ then
\begin{equation}\label{oingoiangoinqah}
|\nabla_{\partial}^l (u_a -u)|=O(r^{1-\tau-l}),\quad\quad\quad l=0,1,2,
\end{equation}
where $u=a_i x^i$. We will also use the notation $\hat{u}=-\mathbf{t}+a_i \mathbf{x}^i=a_i \mathbf{x}^i$ on $\mathbb{R}^3$.

The pre-limit energy of coordinate spheres may be expressed using \eqref{h1} as
\begin{align}\label{oghqopihjpiojwqph}
\begin{split}
E_{\varepsilon}(S_r,\iota_r,u_a)=&\frac{1}{8\pi}\int_{S_r}\left(\sqrt{|\vec{H}_0|^2(|\nabla_\partial u_a|^2 +\varepsilon^2)+(\Delta_\partial u_a)^2} -f_0 \Delta_{\partial} u_a +\alpha_{\frac{\vec{H}_0}{|\vec{H}_0|}}(\nabla_{\partial}u_a)\right)dA\\
&-\frac{1}{8\pi}\int_{S_r}\left(\sqrt{|\vec{H}|^2(|\nabla_\partial u_a|^2 +\varepsilon^2)+(\Delta_\partial u_a)^2} -f_{\varepsilon}\Delta_{\partial} u_a +\alpha_{\frac{\vec{H}}{|\vec{H}|}}(\nabla_{\partial}u_a)\right)dA,
\end{split}
\end{align}
where $\vec{H}=H\nu-(\mathrm{Tr}_{S_r}k)\mathbf{n}$ and $\vec{H}_0=\hat{H}\hat{\nu}$ are the mean curvature vectors of $S_r \subset M$ and $\iota_r(S_r)\subset\mathbb{R}^{3,1}$, and 
\begin{equation}
f_{\varepsilon}=\sinh^{-1}\left(\frac{\Delta_{\partial}u_a}{|\vec{H}|\sqrt{|\nabla_{\partial}u_a|^2 +\varepsilon^2}}\right),\quad\quad\quad
f_0=\sinh^{-1}\left(\frac{\Delta_{\partial}u_a}{|\vec{H}_0|\sqrt{|\nabla_{\partial}u_a|^2 +\varepsilon^2}}\right).
\end{equation}
In the decomposition of the mean curvature vectors, the normals $\nu$ and $\hat{\nu}$ are tangent to $M$ and the constant time slice of Minkowski space respectively.
From \eqref{asymflat}, \eqref{shmin}, and \eqref{oqhgoihqoihioqh} we have
\begin{equation}\label{fohqoihgoiqhoig}
\Delta_{\partial}u_a=-\hat{H}\hat{\nu}(\hat{u})=-|\vec{H}_0|\hat{\nu}(\hat{u})=-|\vec{H}|\hat{\nu}(\hat{u})+O(r^{-1-\tau})
=-H\hat{\nu}(\hat{u})+O(r^{-1-\tau}),
\end{equation}
since $\hat{\nu}(\hat{u})=O(1)$. In particular, observing that $|\nabla_{\partial}u_a|^2 +\hat{\nu}(\hat{u})^2 =1$ produces
\begin{equation}\label{viahgiohaoig}
|\nabla_\partial u_a|^2 +|\vec{H}|^{-2}(\Delta_\partial u_a)^2 =1 +O(r^{-\tau}),\quad\quad\quad
|\nabla_\partial u_a|^2 +|\vec{H}_0|^{-2}(\Delta_\partial u_a)^2 =1 +O(r^{\tau}).
\end{equation}
With the help of $|\nabla_{\partial}u_a|=O(1)$ it follows that
\begin{equation}       
\frac{\left(|\vec{H}_0|+|\vec{H}|\right)(|\nabla_\partial u_a|^2 +\varepsilon^2)}{\sqrt{|\vec{H}_0|^2(|\nabla_\partial u_a|^2 +\varepsilon^2)+(\Delta_\partial u_a)^2}+\sqrt{|\vec{H}|^2(|\nabla_\partial u_a|^2 +\varepsilon^2)+(\Delta_\partial u_a)^2}}
=|\nabla_\partial u_a|^2+O(r^{-\tau} +\varepsilon^2),
\end{equation}
and therefore
\begin{align}\label{alkflangoinqoihgn}
\begin{split}
&\int_{S_r}\left(\sqrt{|\vec{H}_0|^2(|\nabla_\partial u_a|^2 +\varepsilon^2)+(\Delta_\partial u_a)^2} -\sqrt{|\vec{H}|^2(|\nabla_\partial u_a|^2 +\varepsilon^2)+(\Delta_\partial u_a)^2}\right)dA\\
=& \int_{S_r}\left(|\vec{H}_0|-|\vec{H}|\right)\frac{\left(|\vec{H}_0|+|\vec{H}|\right)(|\nabla_\partial u_a|^2 +\varepsilon^2)}{\sqrt{|\vec{H}_0|^2(|\nabla_\partial u_a|^2 +\varepsilon^2)+(\Delta_\partial u_a)^2}+\sqrt{|\vec{H}|^2(|\nabla_\partial u_a|^2 +\varepsilon^2)+(\Delta_\partial u_a)^2}} dA\\
=& \int_{S_r}\left(|\vec{H}_0|-|\vec{H}|\right)|\nabla_\partial u_a|^2  dA  +O(r^{1-2\tau}+\varepsilon^2 r^{1-\tau}).
\end{split}
\end{align}

Consider now the terms involving $f_{\varepsilon}$ and $f_0$. Notice that by setting $\zeta=|\vec{H}|^{-1}|\vec{H}_0|-1=O(r^{-\tau})$ and applying the mean value theorem to the following function of $\zeta$ we find
\begin{align}
\begin{split}
&\sinh^{-1}\left(\frac{\Delta_{\partial}u_a}{|\vec{H}|\sqrt{|\nabla_{\partial}u_a|^2+\varepsilon^2}}\right)\\
=&
\sinh^{-1}\left(\frac{\Delta_{\partial}u_a}{|\vec{H}_0|\sqrt{|\nabla_{\partial}u_a|^2+\varepsilon^2}}\cdot(1+\zeta)\right)\\
=&\sinh^{-1}\left(\frac{\Delta_{\partial}u_a}{|\vec{H}_0|\sqrt{|\nabla_{\partial}u_a|^2+\varepsilon^2}}\right)
+\left[1+\left(\frac{(1+O(r^{-\tau}))\Delta_{\partial}u_a}{|\vec{H}_0|\sqrt{|\nabla_{\partial}u_a|^2+\varepsilon^2}}\right)^2\right]^{-1/2}\!\!\!\!\!\!\!\!\!\!
\frac{\Delta_{\partial}u_a}{|\vec{H}_0|\sqrt{|\nabla_{\partial}u_a|^2+\varepsilon^2}}\cdot \zeta\\
=&\sinh^{-1}\left(\frac{\Delta_{\partial}u_a}{|\vec{H}_0|\sqrt{|\nabla_{\partial}u_a|^2+\varepsilon^2}}\right)
+\frac{\Delta_{\partial}u_a}{|\vec{H}_0|}\left(\frac{|\vec{H}_0|}{|\vec{H}|}-1\right)\left(1+O(r^{-\tau}+\varepsilon^2)\right),
\end{split}
\end{align}
where \eqref{viahgiohaoig} has also been used. Hence \eqref{fohqoihgoiqhoig} yields
\begin{equation}
\int_{S_r}(f_{\varepsilon}-f_0)\Delta_{\partial}u_a dA=\int_{S_r}\hat{\nu}(\hat{u})^2\left(|\vec{H}_0|-|\vec{H}|\right)dA
+O(r^{1-2\tau}+\varepsilon^2 r^{1-\tau}),
\end{equation}
and combining this with \eqref{alkflangoinqoihgn} gives
\begin{align}\label{gohqoighoiqwhoih}
\begin{split}
&\int_{S_r}\left(\sqrt{|\vec{H}_0|^2(|\nabla_\partial u_a|^2 +\varepsilon^2)+(\Delta_\partial u_a)^2} -f_0 \Delta_{\partial} u_a \right)dA\\
&-\int_{S_r}\left(\sqrt{|\vec{H}|^2(|\nabla_\partial u_a|^2 +\varepsilon^2)+(\Delta_\partial u_a)^2} -f\Delta_{\partial} u_a \right)dA\\
=&\int_{S_r}\left(|\vec{H}_0|-|\vec{H}|\right)dA+O(r^{1-2\tau}+\varepsilon^2 r^{1-\tau})\\
=&8\pi\mathcal{E}+o(1)+O(r^{1-2\tau}+\varepsilon^2 r^{1-\tau}).
\end{split}
\end{align}
Here $\mathcal{E}$ is the ADM energy, and in the last step we utilize the fact that the Liu-Yau and Brown-York energy have the same large sphere limit \cite[proof of Theorem 3.1]{WangYaulimit}, together with the convergence of the Brown-York energy to the ADM energy \cite[Theorem 1.1]{FST}.

Lastly, consider the connection 1-forms within \eqref{oghqopihjpiojwqph}. According \eqref{1} we find that the reference connection 1-form vanishes, and with \eqref{2} as well as \eqref{oingoiangoinqah} it follows that
\begin{align}
\begin{split}
\alpha_{\frac{\vec{H}}{|\vec{H}|}}(\nabla_{\partial} u_a)=&-k(\nabla_{\partial}u_a ,\nu)-\nabla_{\partial}u_a \cdot\nabla_{\partial}h\\
=&-k(\nabla_{\partial}u ,\nu)+O(r^{-1-2\tau})-\nabla_{\partial}u_a \cdot\nabla_{\partial}h \\
=&-\!\left( k\!-\!(\mathrm{Tr}_g k)g\right)(\nabla u,\nu)\!-\!(\mathrm{Tr}_{S_r} k)\nu(u)\!+\! h\Delta_{\partial}u_a
\!-\!\mathrm{div}_{\partial}\left(h\nabla_{\partial} u_a\right)\!+\!O(r^{-1-2\tau}),
\end{split}
\end{align}
where
\begin{equation}
h=-\sinh^{-1}\left(\frac{\mathrm{Tr}_{S_r}k}{|\vec{H}|}\right)=-\left(1+O(r^{-\tau})\right)\frac{\mathrm{Tr}_{S_r}k}{|\vec{H}|}
\end{equation}
with the mean value theorem being used in the last equality.
Furthermore \eqref{oqhgoihqoihioqh} implies that $\hat{\nu}(\hat{u})=\nu(u)+O(r^{-\tau})$, which together with  \eqref{fohqoihgoiqhoig} yields
\begin{equation}
\Delta_{\partial}u_a =-|\vec{H}|\nu(u)+O(r^{-\tau}),
\end{equation}
and hence
\begin{equation}
-(\mathrm{Tr}_{S_r} k)\nu(u)+h\Delta_{\partial}u_a= (\mathrm{Tr}_{S_r}k)\nu(u)\cdot O(r^{-\tau})=O(r^{-1-2\tau}).
\end{equation}
Moreover, since $\nabla u=a_i \partial_{x^i} +O(r^{-\tau})$ we obtain
\begin{equation}\label{oaijfoinaoighh}
\int_{S_r}\left(\alpha_{\frac{\vec{H}_0}{|\vec{H}_0|}}(\nabla_{\partial} u_a)-\alpha_{\frac{\vec{H}}{|\vec{H}|}}(\nabla_{\partial} u_a)\right)dA
=-8\pi\langle \mathbf{a},\mathcal{P}\rangle+o(1)+O(r^{1-2\tau}),
\end{equation}
where $\mathcal{P}$ is the ADM linear momentum. By combining \eqref{oghqopihjpiojwqph}, \eqref{gohqoighoiqwhoih}, and \eqref{oaijfoinaoighh} the desired result is achieved
\begin{equation}
\lim_{r\rightarrow\infty}E(S_r,\iota_r,u_a) =\lim_{r\rightarrow\infty}\lim_{\varepsilon\rightarrow 0} 
E_{\varepsilon}(S_r,\iota_r,u_a)=\mathcal{E}-\langle \mathbf{a},\mathcal{P}\rangle.
\end{equation}

\section{First Variation of the Energy}
\label{sec7} \setcounter{equation}{0}
\setcounter{section}{7}

The purpose of this section is to derive the Euler-Lagrange equation for the function $u_a$ at a critical point of the energy, under ideal conditions. In particular, it will be assumed that the critical pair $(\iota,u_a)$ is admissible and has the following properties. The set of critical points for $u_a$ is sufficiently mild to allow for the interchange of limits as $\varepsilon\rightarrow 0$ with integration and variational differentiation, and the Euler characteristics $\chi(\hat{\Sigma}_s)$ of regular level sets within the reference space fill-in $\hat{\Omega}$ take only the value 1. We will use the notation $\delta$ to denote the operation of variation.

Consider the quantity
\begin{align}
\begin{split}
E_{\varepsilon}(\Sigma,\iota,u_a)=&\frac{1}{4}\int_{\underline{u}_a}^{\overline{u}_a}\chi(\hat{\Sigma}_s)ds\\
&-\frac{1}{8\pi}\int_{\Sigma}\left(\sqrt{|\vec{H}|^2 (|\nabla_\partial u_a|^2+\varepsilon^2)+(\Delta_\partial u_a)^2}+\nabla_{\partial}f_{\varepsilon}\cdot\nabla_{\partial}u_a+\alpha_{\frac{\vec{H}}{|\vec{H}|}}(\nabla_{\partial}u_a)\right) dA,
\end{split}
\end{align}
where
\begin{equation}
f_{\varepsilon}=\sinh^{-1}\left(\frac{\Delta_{\partial}u_a}{|\vec{H}|\sqrt{|\nabla_{\partial}u_a|^2 +\varepsilon^2}}\right).    
\end{equation}
Observe that \eqref{oqignoihihniqh} together with Lemma \ref{caseofeq} and the coarea formula show
\begin{equation}
E(\Sigma,\iota,u_a)=\lim_{\varepsilon\rightarrow 0}E_{\varepsilon}(\Sigma,\iota,u_a).
\end{equation}
Direct calculations yield
\begin{equation}
\delta \int_{\underline{u}_a}^{\overline{u}_a}\chi(\hat{\Sigma}_s)ds=\delta \overline{u}_a -\delta\underline{u}_a,
\end{equation}
\begin{equation}
\delta\int_{\Sigma}\sqrt{|\vec{H}|^2(|\nabla_\partial u_a|^2 +\varepsilon^2)+(\Delta_\partial u_a)^2} dA=\int_{\Sigma}\frac{|\vec{H}|^2 \nabla_{\partial}u_a \cdot\nabla_{\partial}\delta u_a +
\Delta_{\partial}u_a \Delta_{\partial}\delta u_a}{\sqrt{|\vec{H}|^2(|\nabla_\partial u_a|^2 +\varepsilon^2)+(\Delta_\partial u_a)^2}} dA,
\end{equation}
and
\begin{equation}
\delta \int_{\Sigma}\alpha_{\frac{\vec{H}}{|\vec{H}|}}(\nabla_{\partial}u_a) dA=\int_{\Sigma}\alpha_{\frac{\vec{H}}{|\vec{H}|}}(\nabla_{\partial}\delta u_a) dA.
\end{equation}
Moreover since
\begin{equation}
\delta f_{\varepsilon}=\frac{\Delta_\partial\delta u_a -(|\nabla u_a|^2 +\varepsilon^2)^{-1}(\Delta_\partial u_a)\nabla_\partial u_a\cdot\nabla_\partial\delta u_a}{\sqrt{|\vec{H}|^2(|\nabla_\partial u_a|^2 +\varepsilon^2)+(\Delta_\partial u_a)^2}},
\end{equation}
we have
\begin{align}
\begin{split}
\delta \int_{\Sigma}\nabla_\partial f_{\varepsilon}\cdot\nabla_{\partial} u_a dA
=&-\delta \int_{\Sigma}f_{\varepsilon}\Delta_{\partial}u_a dA\\
=&-\int_{\Sigma}\left((\delta f_{\varepsilon})\Delta_{\partial}u_a +f_{\varepsilon}\Delta_{\partial}\delta u_a \right)dA\\
=&\int_{\Sigma}\left(\frac{-\Delta_{\partial} u_a\Delta_\partial\delta u_a +(|\nabla u_a|^2 +\varepsilon^2)^{-1}(\Delta_\partial u_a)^2\nabla_\partial u_a\cdot\nabla_\partial\delta u_a}{\sqrt{|\vec{H}|^2(|\nabla_\partial u_a|^2 +\varepsilon^2)+(\Delta_\partial u_a)^2}}\right)dA\\
&+\int_{\Sigma}\nabla_\partial f_{\varepsilon} \cdot\nabla_{\partial}\delta u_a dA.
\end{split}
\end{align}
It follows that
\begin{align}
\begin{split}
\delta E_{\varepsilon}(\Sigma,\iota,u_a)=&\frac{1}{4}(\delta \overline{u}_a -\delta\underline{u}_a)\\ 
&-\frac{1}{8\pi}\int_{\Sigma}\left(|\vec{H}|(\cosh f_{\varepsilon}) \frac{\nabla_{\partial}u_a\cdot\nabla_{\partial}\delta u_a}{\sqrt{|\nabla_{\partial} u_a|^2+\varepsilon^2}}+\nabla_\partial f_{\varepsilon} \cdot\nabla_{\partial}\delta u_a +\alpha_{\frac{\vec{H}}{|\vec{H}|}}(\nabla_{\partial}\delta u_a)\right)dA.
\end{split}
\end{align}

Under the ideal conditions mentioned at the beginning of this section, we may interchange limit and variational derivative, and apply the dominated convergence theorem as in \eqref{ofnqoignoiwnqhoinq} to obtain
\begin{align}
\begin{split}
\delta E(\Sigma,\iota,u_a)=&\lim_{\varepsilon\rightarrow 0}\delta E_{\varepsilon}(\Sigma,\iota,u_a)\\
=&\frac{1}{4}(\delta \overline{u}_a -\delta\underline{u}_a)\\ 
&-\frac{1}{8\pi}\int_{\Sigma}\left(|\vec{H}|(\cosh f) \frac{\nabla_{\partial}u_a\cdot\nabla_{\partial}\delta u_a}{|\nabla_{\partial} u_a|}+\nabla_\partial f \cdot\nabla_{\partial}\delta u_a +\alpha_{\frac{\vec{H}}{|\vec{H}|}}(\nabla_{\partial}\delta u_a)\right)dA,
\end{split}
\end{align}
where
\begin{equation}
f=\sinh^{-1}\left(\frac{\Delta_{\partial}u_a}{|\vec{H}||\nabla_{\partial}u_a|}\right).    
\end{equation}
If the variations $\delta u_a$ are plentiful enough so as to include all smooth functions, then we find that the critical isometric embedding pair gives rise to a weak solution of the 4th order equation
\begin{equation}
\text{div}_{\sigma}\left(|\vec{H}|(\cosh f)\frac{\nabla_\partial u_a}{|\nabla_\partial u_a|}+\nabla_{\partial} f+V\right)=2\pi(\pmb{\delta}_- -\pmb{\delta}_+),
\end{equation}
in which $V$ is the dual vector field to the connection 1-form $\alpha_{\frac{\vec{H}}{|\vec{H}|}}$ and $\pmb{\delta}_{\pm}$ are Dirac delta distributions at the max and min points.

\end{document}